\documentclass[12pt,a4paper]{article}

\usepackage{amsmath,amstext,amssymb,amscd,euscript}
 \usepackage{graphicx}

\usepackage[english]{babel}
\usepackage[cp1251]{inputenc}

\newcommand{\Xcomment}[1]{}
\newtheorem{theorem}{Theorem}[section]
\newtheorem{lemma}[theorem]{Lemma}
\newtheorem{corollary}[theorem]{Corollary}
\newtheorem{prop}[theorem]{Proposition}

\newenvironment{proof}{\noindent{\bf Proof}\/}%
{\hfill$\qed$\medskip}

\def\qed{\Box}

\makeatletter \@addtoreset{equation}{section} \makeatother

\newenvironment{numitem1}{\refstepcounter{equation}\begin{enumerate}%
\item[(\thesection.\arabic{equation})]}{\end{enumerate}}

\newcommand{\refeq}[1]{(\ref{eq:#1})}  

\def\rest#1{_{\,\vrule height 1.6ex width 0.05em depth 0pt\, #1}}

 \makeatletter
\renewcommand{\section}{\@startsection{section}{1}{0pt}%
{-3.5ex plus -1ex minus -.2ex}{2.3ex plus .2ex}%
{\normalfont\Large}}
 \makeatother

 \makeatletter
\renewcommand{\subsection}{\@startsection{subsection}{2}{0pt}%
{-3.0ex plus -1ex minus -.2ex}{1.5ex plus .2ex}%
{\normalfont\normalsize\bf}}
 \makeatother

 \newcommand{\SEC}[1]{\ref{sec:#1}}  

\def\Rset{{\mathbb R}}

\def\Ascr{{\cal A}}
\def\Bscr{{\cal B}}

\def\Escr{{\cal E}}

\def\Gscr{{\cal G}}

\def\Sscr{{\cal S}}

\def\Vscr{{\cal V}}

\def\tilde{\widetilde}

\def\rest#1{_{\,\vrule height 1.5ex width 0.05em depth 0pt\, #1}}

\begin{document}

\baselineskip=15pt
\parskip=3pt

\title{On diversifying stable assignments}

\author{Alexander V.~Karzanov%
\thanks{Central Institute of Economics and Mathematics of
the RAS, 47, Nakhimovskii Prospect, 117418 Moscow, Russia; email:
akarzanov7@gmail.com.}
}
  
\date{}
 
 \maketitle

\begin{abstract}
We consider the stable assignment problem on a graph with nonnegative real capacities on the edges and quotas on the vertices, in which the preferences of agents are given via \emph{diversifying} choice functions. We prove that for any input of the problem, there exists exactly one stable assignment, and propose a polynomial time algorithm to find it.
 \medskip
 
 \end{abstract}


\section{Introduction}  \label{sec:intr}

Starting from the classical work by Gale and Shapley~\cite{GS} on stable marriages in bipartite graphs, there have appeared an immense number of researches of many authors devoted to various models of stability on graphs, and wider. Seemingly one of the most general stability problems on bipartite graphs was introduced and well studied by Alkan and Gale~\cite{AG} who, in the framework of their model, proved the existence of stable solutions and other nice properties. 

In their model, the preferences of agents (identified with vertices of the graph) are defined by use of choice functions of a rather wide spectrum (merely obeying the standard rules of consistence, persistence, and size monotonicity). To illustrate this setting, two particular examples are exposed (Examples~1 and ~2 in Section~2 of~\cite{AG}), both dealing with nonnegative real-valued capacities on the edges and quotas on the vertices. The first one is the well-known stable \emph{allocation} model introduced by Baiou and Balinski~\cite{BB}, and the second one involves the so-called \emph{diversifying} choice functions.

In this note we just consider the second, diversifying, model and show that for any graph (including non-bipartite ones), capacities and quotas, there exists a unique stable solution (which looks somewhat surprising). We also develop a strongly polynomial algorithm to find this solution.

The content is organized as follows. Section~\SEC{prelim} contains basic definitions and reviews known results needed to us. Section~\SEC{main} states the  result on uniqueness for diversifying two-sided markets and gives a proof (which is relatively short). An efficient algorithm to find a stable assignment in the bipartite case is described in Section~\SEC{begin}. The final Section~\SEC{general} is devoted to a general (non-bipartite) case.


\section{Preliminaries}  \label{sec:prelim}

Throughout we consider the \emph{diversifying} model of stable assignments as exposed in Example~2 from Sect.~2 in Alkan-Gale's paper~\cite{AG}.

We start with definitions and settings. We are given a bipartite graph $G=(V,E)$ with parts (color classes)  $F$ and $W$, conditionally called \emph{firms} and \emph{workers}, respectively. The edges $e\in E$ are equipped with \emph{capacities} (upper bounds) $b(e)\in\Rset_+$, and the  vertices $v\in V$ with \emph{quotas} $q(v)\in\Rset_+$. For a vertex $v\in V$, the set of its incident edges is denoted by $E_v$. In general, we do not impose any relations between capacities and quotas, such as $b(E_v)\ge q(v)$ for $v\in V$ or $q(F)=q(W)$. (Hereinafter, for a function $f:S\to \Rset$ and a finite subset $S'\subseteq S$, $f(S')$ denotes the sum $\sum(f(e)\colon e\in S')$. Also when $S$ is finite, we denote $\sum(|f(e)|\colon e\in S)$ as $|f|$.)

The set of \emph{admissible assignments} in the model is defined to be the full box $\Bscr:=\{x\in\Rset_+^E\colon x\le b\}$. For a vertex (``agent'')  $v\in V$, the restriction of an assignment $x\in \Bscr$ to the set $E_v$ is denoted as $x_v$, and the set of these restrictions is denoted as $\Bscr_v$. On each of these sets, there is a \emph{choice function} (CF)  $C_v:\Bscr_v\to\Bscr_v$ which acts as follows: for each $z\in \Bscr_v$, there holds $C_v(z)\le z$, and
  \begin{numitem1} \label{eq:z-q}
if $|z|\le q(v)$, then $C_v(z)=z$, and if $|z|>q(v)$, then one takes the number  $r=r^z$ (the \emph{cutting height}) such that $\sum(r\wedge z(e)\colon e\in E_v)=q(v)$, and one puts
  $$
  C_v(z)(e):=r\wedge z(e), \quad e\in E_v.
  $$
  \end{numitem1}
A CF $C=C_v$ of this sort is called \emph{diversifying} in~\cite{AG}. It is noted there (and this is not difficult to check) that it satisfies the conditions of \emph{consistence} (saying that for admissible $z,z'\in \Bscr_v$, if $C(z)\le z'\le z$ then $C(z')=C(z)$) and \emph{persistence} (saying that $z'\le z$ implies $C(z)\wedge z'\le C(z')$). (Hereinafter, for $x,y\in \Rset^S$, $x\wedge y$ denotes the coordinate-wise ``meet'', taking the values $\min(x(e),y(e))$ for $e\in S$, and $x\vee y$ denotes the ``join'', taking the values $\max(x(e),y(e))$.)  As a consequence, $C$ has the property of \emph{stationarity} $C(z\vee z')=C(C(z)\vee z')$ for any $z,z'\in \Bscr_v$, in particular, implying  that $C(C(z))=C(z)$. Also $C$ satisfies a strengthened version of  \emph{size-monotonicity} (where $z\le z'$ implies $|z|\le|z'|$), namely, the \emph{quota-filling} one (saying that $|C(z)|=q(v)$ if $|z|\ge q(v)$, and $C(z)=z$ otherwise).

The assignments $z\in\Bscr_v$ with $z=C(z)$ are called \emph{rational} and the set of these is denoted by $\Ascr_v$. In accordance with the general Alkan-Gale's model, we say that $z\in\Ascr_v$ is (revealed) \emph{preferred} to $z'\in\Ascr_v$, and denote this as  $z\succeq z'$, if 
  $$
  C(z\vee z')=z.
  $$
This relation is transitive and determines a lattice on  $\Ascr_v$ with the join operation $\curlyvee$ of the form $z\curlyvee z'=C(z\vee z')$.

Let us specify the relation $\succeq$ in our case. We call $z\in\Ascr_v$ \emph{fully filling} if $|z|=q(v)$. For such a $z$, we partition $E_v$ into two subsets: the \emph{head} (arg-maximum) $H(z):=\{e\in E_v\colon  z(e)=r^z\}$ and the \emph{tail} (the rest) $T(z):=E_v-H(z)$, where $r^z=\max(z(e)\colon e\in E)$. In case $|z|<q$, we formally define $H(z):=\emptyset$ and $T(z):=E_v$. One can see that
  \begin{numitem1} \label{eq:zzp}
for $z,z'\in\Ascr_v$, when $z$ if fully filling, $z\succeq z'$ holds if and only if  $z(e)\ge z'(e)$ for all $e\in T(z)$; if, in addition, $z'$ is fully filling either, then $r^z\le r^{z'}$ and $H(z)\supseteq H(z')$.
 \end{numitem1}
(Indeed, taking into account that $|z|=q(v)\ge |z'|$, one can observe from~\refeq{z-q} that $C(z\vee z')=z$ implies $z\rest{T(z)}\ge {z'}\rest{T(z)}$, and vice versa.)
 \medskip
 
\noindent\textbf{Definition.} An edge $e\in E_v$ is called \emph{satiated} (or \emph{non-interesting}) for the agent $v$ with respect to an assignment $z\in\Ascr_v$ if $e\in H(z)\cup U(z)$, where $U(z)$ denotes the set of edges $e\in E_v$ with the attained upper bound: $z(e)=b(e)$. 
 \medskip
 
\noindent(In other words, $e$ is interesting for  $v$ (from the viewpoint of a possible replacement of $z$ by a more preferred assignment) if $e$ is contained in the tail $T(z)$ and, at the same time, $z(e)$ is lower than the upper bound.)
 
Now from individual agents we come to assignments on the whole edge set $E$. An admissible assignment $x\in\Bscr$ is said to be rational if this is so for all agents, i.e. $x_v\in\Ascr_v$ for all $v\in V$; the set of these assignments is denoted by $\Ascr$. 
\medskip

\noindent\textbf{Definitions.} Let $x\in\Ascr$. An edge $e=ij\in E$ is said to be \emph{blocking} for $x$ if $e$ is non-satiated for both agents $i$ and $j$ (when considering the restrictions $x_i$ on $\Ascr_i$ and $x_j$ on $\Ascr_j$). When there is no blocking edge for $x$, the assignment $x$ is called \emph{stable}.
 \medskip
 
The set of stable assignments is denoted by $\Sscr=\Sscr(G,b,q,C)$.
In our case the stability of $x$ is specified as follows:
  \begin{numitem1} \label{eq:stabx}
for an edge $e=ij\in E$, if $e\in T(x_i)-U(x_i)$ (that is $e$ is non-satiated for $i$), then $x_j$ is fully filling and $e\in H(x_j)$ (whence $e$ is satiated for $j$), and similarly by swapping $i$ and $j$.
 \end{numitem1}
 
For stable $x,y\in\Sscr$, if the preference relation $x_v\succeq y_v$ takes place for all firms $v\in F$ (all workers $v\in W$), we write $x\succeq_F y$ (respectively,  $x\succeq_W y$). The general theorem on stability for two-sided markets by Alkan and Gale in~\cite{AG} can be specified for our case with diversifying choice functions of all agents as follows:
  \begin{numitem1} \label{eq:stab-div}
the set $\Sscr$ of stable assignments is nonempty, and the relation $\succeq_F$ determines a lattice structure on $\Sscr$ having maximal and minimal elements; in addition:
 \begin{itemize}
\item[(a)] the order $\succeq_F$ is inverse to the order $\succeq_W$; namely, if  $x\succeq_F y$, then $y\succeq_W x$, and vice versa;
\item[(b)] for each $v\in V$, the size of a stable assignment at $v$ is invariant; that is, the value $|x_v|$ is the same for all $x\in \Sscr$.
 \end{itemize}
 \end{numitem1}
 
Note that property~(b) has a refinement  (Corollary~3 in~\cite{AG}), which will be useful for us; namely:
 \begin{numitem1} \label{eq:constant}
for an agent $v\in V$, if CF $C_v$ is quota-filling, and if a stable assignment $x$ is not fully filling at $v$ (i.e. $|x_v|<q(v)$), then all stable assignments coincide within $E_v$.
 \end{numitem1}

\section{Theorem} \label{sec:main}

In this section we show the following
  \begin{theorem} \label{tm:unique}
In the diversifying two-sided (bipartite) market as above, the set $\Sscr$ of stable assignments consists of a unique element.
  \end{theorem}
  \begin{proof}
~An edge connecting vertices $i\in F$ and $j\in W$ will be denoted as $ij$. 

We know that $\Sscr$ is nonempty (by~\refeq{stab-div}). Suppose that $|\Sscr|> 1$. Then there are two comparable stable assignments $x\succ_F y$ (for example, the most and least preferred ones for $F$). Define 
  $$
  E^+:=\{e\in E\colon x(e)<y(e)\}\quad\mbox{and}\quad  E^-:=\{e\in E\colon x(e)>y(e)\}.
  $$
  
From~\refeq{constant} it follows that each edge in $E^+\cup E^-$ connects two fully filled vertices. 

Choose an edge $e=ij\in E^+$. By the invariance property~\refeq{stab-div}(b) applied to the vertices $i$ and $j$, there exist edges $e'=ij'\in E^-$ and $e''=i''j\in E^-$. Considering~\refeq{zzp} for the restrictions $z=x_i$ and $z'=y_i$ (which satisfy $x_i\succ y_i$) and the edges $e$ and $e'$, we can conclude from the inequalities $x(e)<y(e)$ and $x(e')>y(e')$ that 
  \begin{equation} \label{eq:epe}
  x(e')\le x(e)\quad \mbox{and} \quad y(e')<y(e).
  \end{equation}
Similarly, applying~\refeq{zzp} to the functions $z=y_j$ and $z'=x_j$ (satisfying $y_j\succ x_j$, in view of $y\succ_W x$, by the polarity  property~\refeq{stab-div}(a)) and to the edges $e$ and $e''$, we can conclude from $x(e)<y(e)$ and $x(e'')>y(e'')$ that
  \begin{equation} \label{eq:eepp}
  x(e)< x(e'')\quad \mbox{and} \quad y(e)\le y(e'').
  \end{equation}

The sequence $e',e,e''$ can be extended in both directions as long as wished. As a result, one can extract a cycle
 $$
 e_1,e_2,\ldots,e_i,\ldots,e_{2k},e_{2k+1}=e_1,
 $$
in which the edges with odd indices belong to $E^-$, and the ones with even indices belong to $E^+$, and, in addition, for each $i$, the common vertex for  $e_i$ and $e_{i+1}$ is contained in the part $F$ for $i$ odd, and in $W$ for $i$ even. Then, using inequalities as in~\refeq{epe} and~\refeq{eepp}, we obtain the sequence of inequalities
  $$
  x(e_1)\le x(e_2)<x(e_3)\le x(e_4)<\cdots \le x(e_{2k})< x(e_{2k+1}=e_1),
  $$
which leads to a contradiction, yielding the result.
  \end{proof}
 

\section{An efficient construction}  \label{sec:begin}

The basic method of finding a stable assignment for our diversifying model (which is unique by Theorem~\ref{tm:unique}) developed in this section resembles the method of proving the existence of a stable solution in a general case presented in~\cite[Sect.~3.1]{AG}. It consists of a sequence of iterations, and each iteration is divided into two \emph{phases}. All functions on $E$ arising during the process are admissible (bounded by the capacities $b$) but need not satisfy some quotas in the part $W$. 

At the \emph{1st phase} of each iteration, a current function $x$ on $E$ is updated by use of transformations carried out independently in the restrictions $x_i=x\rest{E_i}$ for all $i\in F$.

We describe such a transformation, considering a firm $i\in F$ and denoting the current admissible function on $E_i$ as $z$.  For this $z$, we assume that the following hold: 
 \begin{numitem1} \label{eq:zL}
$|z|\le q(i)$, and in the set $E_i$, a certain subset $L=L_i$ is distinguished.
  \end{numitem1}
 
The procedure preserves the values of $z$ within $L$ and updates $z$ on $E_i-L$ toward increasing by the following rules: first we assign $y=y_i\in\Rset^{E_i-L}$  according to the upper bounds:  $y(e):=b(e)$ for all $e\in E_i-L$, and then we cut $y$ at the maximal height $r$ that provides that the quota $q(i)$ on $E_i$ is not exceeded; namely (cf.~\refeq{z-q}):
  \begin{numitem1} \label{eq:cut-y}
take $r$ such that $\sum(r\wedge y(e)\colon e\in E_i-L)=q(i)-z(L)$, and put $\tilde z(e):= r\wedge y(e)$ for all $e\in E_i-L$ and $\tilde z(e):=z(e)$ for all $e\in L$.
 \end{numitem1}
 
The obtained function $\tilde z$ either is fully filling: $|\tilde z|=q(i)$, or not: $|\tilde z|<q(i)$, and in the latter case the equality $\tilde z(e)=b(e)$ holds for all $e\in E_i-L$.

Taking together the functions $\tilde z=\tilde z_i$ over $i\in F$, we obtain a new admissible function $\tilde x$ on the whole $E$. It satisfies $\tilde x\ge x$, but for some vertices $j$ in the part $W$, the quota may be exceeded, i.e. $|\tilde x_j|>q(j)$ may happen.

The arisen excesses are eliminated at the \emph{2nd phase} of the iteration. More precisely, for each vertex $j\in W$ such that $|\tilde x_j|>q(j)$, we properly reduce (cut down) the values of  $\tilde x_j$ according to the quota $q(j)$ (acting as in~\refeq{z-q}). As a result, we obtain an admissible function $x'$ on $E$ satisfying the quotas $q(v)$ for all $v\in V$; it coincides with $\tilde x$ on the subsets $E_j$ ($j\in W$) where $|\tilde x_j|\le q(j)$ took place before, and is fully filling for the other $j\in W$. However, $x'$ may be non-stable. For the purposes of the next iteration, the following addition actions are fulfilled:
  \begin{numitem1} \label{eq:addL}
If, due to the cutting operation, a decrease in some edge $e=ij$ happened, i.e.  $x'(e)<\tilde x(e)$ (and therefore, $e$ is added to the head $H(x'_j)$), then the edge $e$ is inserted in the current set $L_i$ (if it was not inserted in it on preceding iterations).
  \end{numitem1}
 
In the beginning of the process, we put $x:= 0$ and $L_i:=\emptyset$ for all $i\in F$. Then at the 1st phase of the 1st iteration, all intermediate functions  $y_i$ attain the upper bounds, and the phase constructs an admissible assignment $\tilde x$ on $E$ satisfying the quotas for all $i\in F$ and such that 
  \begin{numitem1} \label{eq:init-x}
for each $i\in F$, each edge $e$ in the tail $T(\tilde x_i)$ is satiated ($\tilde x(e)=b(e)$). 
\end{numitem1} 

(E.g., in the special case with $b(e)\ge q(i)/|E_i|$ for all $e\in E_i$, the cutting height $r$ is just equal to the number $q(i)/|E_i|$, and the head $H(\tilde x_i)$ embraces the whole $E_i$.)

The process terminates as soon as the function $\tilde x$ constructed at the 1st phase of the current iteration becomes satisfying the quotas $q(j)$ for all $j\in W$. In order to analyze  the convergence of the process and show the stability of the resulting function, we need two lemmas. These lemmas deal with two consecutive iterations, say, $k$th and $(k+1)$th ones, denote by $x$ and $\tilde x$ the functions in the beginnings of the 1st and 2nd phases of $k$th iterations, respectively, and denote by $x'$ and $\tilde x'$ similar functions for $(k+1)$th iteration.
 \begin{lemma} \label{lm:xpj-xj}
For each $j\in W$, there holds $x'_j\succeq x_j$.
 \end{lemma}
   \begin{proof}
This follows from the fact that the 1st phase of $k$th iteration constructs $\tilde x\ge x$ (the usual coordinate-wise comparison), and therefore, under the cutting operation with $\tilde x_j$ at the 2nd phase (in case $|\tilde x_j|> q(j)$), we obtain $x'(e)=\tilde x(e)\ge x(e)$ for any $e\in T(x')$ (which implies $H(x'_j)\supseteq H(x_j)$, in view of $|x'_j|=q(j)\ge|x_j|$).
  \end{proof}
  \begin{lemma} \label{lm:LjLj}
Suppose that in the beginning of $k$th iteration, for all $i\in F$, each edge $ij\in L_i$ belongs to the head $H(x_j)$. Then a similar property is valid for $(k+1)$th iteration as well.
  \end{lemma}
  \begin{proof}
Let us denote the corresponding $L$-sets for $(k+1)$th iteration with primes (keeping the previous notation for $k$th iteration). Under the construction, if $e=ij\in L_i$, then $e\in L'_i$. By the supposition in the lemma, we have $e\in H(x_j)$. Also there holds $H(x'_j)\supseteq H(x_j)$ (in view of $x'_j\succeq x_j$, by Lemma~\ref{lm:xpj-xj}). Therefore, $e\in H(x'_j)$. And if $e=ij\in L'_i-L_i$, then at the 2nd phase of $k$th iteration, the current function decreases at $e$, implying that $e$ must be added to $H(x'_j)$.
  \end{proof}
  \begin{corollary} \label{cor:stab}
In the assumption that the process terminates in finite time, the resulting function $\tilde x$ is a stable assignment.
  \end{corollary}
 \begin{proof} ~When the process terminates, the current function $\tilde x$ is an admissible assgnment satisfying the quotas for all vertices in $G$. We consider an arbitrary edge $e=ij\in E$ and show that it is not blocking for $\tilde x$.

To show this, we may assume that $e$ does not belong to the head $H(\tilde x_i)$ (otherwise $e$ is satiated for $j$ and we are done). Then from the construction of $\tilde x$ it follows that $e$ belongs to the current set $L_i$, or $\tilde x(e)$ attains the upper bound $b(e)$. In the latter case, $e$ is satiated w.r.t. $\tilde x$, and in the former case, applying Lemma~\ref{lm:LjLj}, step by step, through the sequence of iterations (starting from the 1st one, in which $L_j=\emptyset$), we obtain that $e$ belongs to the head $H(x_j)$, where $x$ is the function in the beginning of the 1st phase of the last iteration. Also $|x_j|=q(j)$ must hold, which implies $\tilde x_j=x_j$ (since the process terminates and the 1st stage does not decrease the current function). Therefore, $e\in H(\tilde x_j)$. 

Now assume that the edge $e$ does not belong to the head $H(\tilde x_j)$, and there holds $\tilde x(e)<b(e)$. Considering the last moment when the current function changes in  $E_i$, one can conclude that this happens at the 1st phase of some iteration, and moreover, at this moment, $e$ does not belong to the current set $L_i$. As a consequence, $e$ should be in the head $H(\tilde x_j)$.

Thus, in all cases $e$ is not blocking for $\tilde x$.      
\end{proof}

Now we are going to investigate the convergence of the process. In reality it may happen that the number of iterations is infinite; in this case the constructed functions converge to a limit, and one can see that this limit function is a stable assignment. Nevertheless, we can modify the process so as to obtain a finite algorithm with the number of iterations polynomial in $|E|$.

Before doing so, let us analyze the above method. We call an iteration \emph{positive} if at least one of the following events happens in it: (a) for some $i\in F$, the subset $L_i$ increases; (b) for some $j\in W$, the size (number of edges) $h_j$ of the head $H_j$ increases; (c) the number $\Pi$ of vertices $j\in W$ which attain the quota $q(j)$ increases (it is clear that after that the fully filling property for $j$ will be maintained).One can see that each of these parameters is monotone non-decreasing during the process, and therefore, the number of positive iterations does not exceed $2|E|+|W|$. 

Now consider a sequence $Q$ of non-positive iterations going in series. As before, we denote the current set of fully filled vertices in $W$ by $W^+$.

At the 1st phase of a current iteration among these, the sum of values of the current function over the edges increases by some positive amount  $\omega$. Note that if for some vertex $j\in W^+$, the value of the function increases on some edge $e=ij$, then this edge cannot belong to the head $H_j$ at the preceding iteration. For otherwise,  as a result of the cutting operation for $j$ at the 2nd phase of the current iteration, the value on $e$ decreases, and $e$ is added as a new element to the set $L_i$ (which is impossible since the iteration is non-positive). It follows that an increase in an edge $e=ij$ at the 1st phase of an iteration in $Q$ is possible only if $e$ is contained in the current tail $T_j$. At the same time, $j$ may belong to $W-W^+$, but if the current iteration is not the last one in $Q$, then in the tails of vertices in $W^+$, the total value must be increased by some positive amount $\omega'$. Obviously, $\omega'\le \omega$ (the equality may happen only if there is no contribution from $\omega$ to the vertices in $W-W^+$). To preserve the current set $W^+$ of fully filled vertices at the 2nd phase of the current iteration under the corresponding cutting operations, the total value in the heads $H_j$ of these vertices $j$ must be decreased by the same amount $\omega'$  (which in turn leads to an increase of ``underloading'' to the quotas in the vertices of $F$ by the same amount $\omega'$);
at the same time, these heads preserve, as well as the sets $L_i$. As a result, we come to the 1st phase of the next iteration with a smaller ``quota-underloading'' $\omega'$, and when such ``underloadings'' decrease geometrically, the sequence  $Q$ can continue infinitely (yet providing a convergence to a limit function, which yields a stable assignment).

Now using a certain aggregation, we can provide a finite convergence. For this aim, while preserving the above parameters $|L_i|$, $h_j$, $\Pi$, we introduce additional parameters. These are: (d) the sizes (numbers of edges) $h'_i$ of the heads $H_i$ of vertices $i$ in the current set $F^+$ of fully filled vertices in the part $F$; (e) the number $\rho$ of edges $e=ij$ in the tails $T_j$ for which the upper bound $b(e)$ is attained. One can see that, under preserving $|L_i|$, $h_j$, $\sigma$, the values $h'_i$ are monotone non-increasing, while $\rho$ is monotone non-decreasing. Therefore, the number of changes of such parameters in $Q$ is at most $2|E|$. We call a maximal subsequence in $Q$ with no changes of this sort \emph{homogeneous}.

Our trick consists in replacing a homogeneous subsequence $Q'$ by one ``big'' (aggregated) iteration. This is done as follows. In the beginning of  $Q'$, we distinguish two sets of edges: the set $A^+$ where the values increase at the 1st phase (then the edges in it go from the heads in $F$ to the tails in $W$), and the set $A^-$ where the values decrease at the 2nd phase (then the edges in it go to the heads in $W$). Denote the vector of the corresponding changes on $A^+\cup A^-$ during both phases of the first iteration in $Q'$ as $\delta$, and determine (which is rather straightforward) the maximal positive number $\xi$ for which the change by $\xi\delta$ is correct. More precisely, under the change, we should not exceed the bounds $b(e)$ of edges $e$ from $A^+$, as well as the quotas in the vertices of $F^+$ and $W-W^+$ (as to the vertices in $W^+$ and $F-F^+$, the first ones continue to be fully filled, and the values on edges for the second ones can merely decrease, but not increase, thus preserving the quota restrictions). 

One can check that the resulting function on $E$ obtained on a big iteration is well defined, that at least one of the above-mentioned five parameters changes, and that a big iteration can be implemented in $O(|E|)$ time. Thus, we obtain
 \begin{prop} \label{pr:modif}
The modified algorithm is finite, consists of $O(|E|^2)$ iteration (where each takes  $O(|E|)$ time), and finds a stable assignment for the diversifying two-sided market in question.
  \end{prop}
  
Also since a stable assignment is unique, by Theorem~\ref{tm:unique}, we can conclude that the stable assignment constructed by the algorithm is optimal for both sides $F$ and $W$.


\section{Diversifications in a general graph} \label{sec:general}

The diversifying model on a two-sided (bipartite) market can be generalized in a natural way to an arbitrary graph $G=(V,E)$ equipped with capacities $b(e)\in\Rset_+$ on the edges $e\in E$ and quotas $q(v)\in\Rset_+$ on the vertices $v\in V$, and with the choice functions $C_v$ on subsets $E_v$, $v\in V$, that are defined, as before, by rule~\refeq{z-q}. Here all definitions of a local character (concerning sets $E_v$ or set pairs $\{E_u,E_v\}$ for edges $\{u,v\}\in E$) are analogous to those in the bipartite case.

A popular method for problems on stability (and wider) for non-bipartite graphs consists in a reduction to a symmetric, or ``self-dual'', bipartite graph by use of splitting (doubling) vertices and edges (as an example, one can mention the reduction of the stable roommates problem (introduced and studied by Irving~\cite{irv}) to the classical stable marriage problem~\cite{GS} as demonstrated in~\cite{hsu}).

Let us briefly describe a reduction of this sort in our case. Each vertex $v\in V$ generates two copies $v_1$ and $v_2$, and each edge $\{u,v\}\in E$ two edges $u_1v_2$ and $v_1u_2$. As a result, we obtain bipartite graph $\Gscr=(\Vscr,\Escr)$ with vertex parts $V_k=\{v_k\colon v\in V\}$, $k=1,2$. Denote the natural symmetry (involution) on $\Vscr\cup\Escr$ by $\sigma$; namely, $\sigma(v_k)=v_{3-k}$, $k=1,2$, and $\sigma(u_1v_2)=v_1u_2$. The capacities and quotas for $G$ generate symmetric capacities and quotas for $\Gscr$; namely, $b(u_1v_2):=b(\{u,v\})$ for $\{u,v\}\in E$ and $q(v_k):=q(v)$ for $v\in V$. Accordingly, an (admissible, rational, stable) assignment $x:\Escr\to\Rset_+$ is called symmetric if $x(e)=x(\sigma(e))$ for all $e\in\Escr$.

Denote by $\pi$ the natural projection that sends each edge  $u_1v_2\in \Escr$ to the edge $\{u,v\}\in E$. The following correspondence is easy:
  \begin{numitem1} \label{eq:biject}
the map $x(e)\mapsto \tilde x(e')$, where $e\in E$ and $e'\in\pi^{-1}(e)$, establishies a bijection between stable assignments $x$ for $G$ and symmetric stable assignment $\tilde x$ for $\Gscr$.
  \end{numitem1}

By Theorem~\ref{tm:unique}, $\Gscr$ has a unique stable assignment $x'$. Form the function $x''$ symmetric to $x'$, i.e. satisfying $x''(e)=x'(\sigma(e))$ for all $e\in \Escr$. Obviously, $x''$ is a stable assignment as well. Therefore, in light of the uniqueness, $x''$ must coincide with $x'$. This means that $x'$ is (self)symmetric, and we obtain
  \begin{corollary} \label{cor:unique}
For any graph $G=(V,E)$, capacities $b:E\to\Rset_+$, quotas $q:V\to\Rset_+$, and diversifying choice functions as above, there exists exactly one stable assignment.
  \end{corollary}
  
An example of reducing to the bipartite case is illustrated in Fig.~\ref{fig:exam}. Here the left fragment shows the initial graph $G$ with vertices $a,b,c$; the numbers by vertices indicate their quotas, the left numbers by edges indicate their capacities, while the right numbers indicate the values of stable assignment. The right fragment shows the corresponding bipartite graph $\Gscr$. Note that the vertices $a,b$  are fully filled, whereas $c$ is not.

 \begin{figure}[htb]
\begin{center}
\vspace{-0cm}
\includegraphics[scale=0.8]{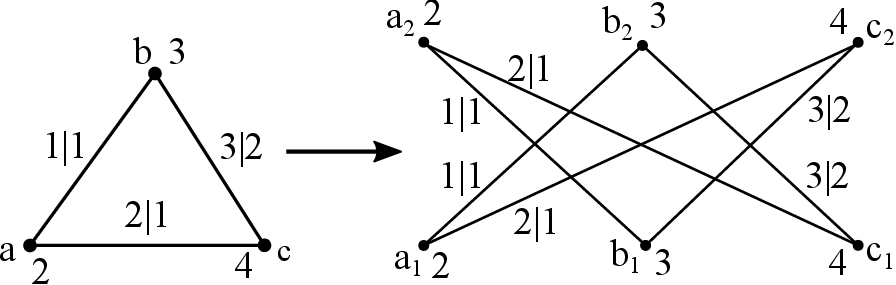}
\end{center}
\vspace{-0.3cm} 
\caption{Reduction to the bipartite case}
 \label{fig:exam}
\end{figure}

\textbf{A note of August 28, 2023:} Yesterday Vladimir Danilov told me a proof of a generalization on hypergraphs. More precisely, for a finite hypergraph $H=(V,E)$ with hyper-edge capacities (upper bounds) $b\in\Rset^E_+$, vertex quotas $q\in\Rset^V_+$, and diversifying choice functions $C_v$ on the sets $E_v$ of hyper-edges incident to vertices $v\in V$ (defined similarly to~\refeq{z-q}), one shows the existence of a unique stable assignment $x$ on $E$. An idea of proof is as follows. For each vertex $v\in V$, take the capacity restriction $b_v:=b\rest{E_v}$, define $z_v:=C_v(b_v)$, and let $m_v$ be the minimum of values $z_v(e)$ among $e\in E_v$. Let $m:=\min\{m_v\colon v\in V\}$ and choose a hyper-edge $e_0$ for which $z_v(e_0)=m_v=m$, where $v$ is (some) vertex incident with $e_0$. The following key property is proved (which is not difficult): for any stable assignment $x$ for $(H,b,q,C)$, the value $x(e_0)$ is equal to $m$. This provides a recursive method: namely: reduce $H$ to $H'=(V,E-\{e_0\})$ and update the quotas $q(v)$ for the vertices $v$ incident with $e_0$ as $q'(v):=q(v)-m$. By induction, the reduced problem has a unique stable solution $x$. One shows that extending $x$ to $e_0$ as $x(e_0):=m$, we obtain a stable assignment for the initial hypergraph. Note that the number of iterations (recursion steps) in the process is $O(|E|)$, and each iteration takes $O(|V| |E|)$ operations.

\end{document}